\documentclass[12pt]{amsart}
\usepackage{amssymb, amsfonts, amsthm}
\usepackage{graphicx}
\usepackage{float}
\newtheorem{theorem}[equation]{Theorem}
\newtheorem{lemma}[equation]{Lemma}
\newtheorem{corollary}[equation]{Corollary}

\numberwithin{equation}{section}

\theoremstyle{definition}
\newtheorem{definition}[equation]{Definition}
\newtheorem{remark}[equation]{Remark}

\newcounter{minutes}
\divide\time by 60
\newcounter{hours}
\multiply\time by 60 \addtocounter{minutes}{-\time}

\begin{document}

\title[Geometric properties of basic hypergeometric functions
 ]%
{Geometric properties of basic hypergeometric functions}

\def\thefootnote{}
\footnotetext{ \texttt{\tiny File:~\jobname .tex,
          printed: \number\day-\number\month-\number\year,
          \thehours.\ifnum\theminutes<10{0}\fi\theminutes}
} \makeatletter\def\thefootnote{\@arabic\c@footnote}\makeatother

\author{Sarita Agrawal
}
\address{Sarita Agrawal, Discipline of Mathematics,
Indian Institute of Technology Indore,
Indore 452 017, India}
\email{saritamath44@gmail.com}

\author{Swadesh Sahoo$^\dagger $}
\address{Swadesh Sahoo, Discipline of Mathematics,
Indian Institute of Technology Indore,
Indore 452 017, India}
\email{swadesh@iiti.ac.in}

\keywords{univalent, starlike and close-to-convex functions, convexity
in the direction of the imaginary axis, continued fraction, $g$-fraction, Hausdorff moment sequence, 
$q$-difference operator, Gauss and basic hypergeometric functions. \\
$
^\dagger$ {\tt Corresponding author}
}

\begin{abstract}
In this paper we consider basic hypergeometric functions introduced by Heine.
We study mapping properties of certain ratios of basic hypergeometric functions 
having shifted parameters and show that they map the domains of analyticity
onto domains convex in the direction of the imaginary axis. 
In order to investigate these mapping properties, few useful identities
are obtained in terms of basic hypergeometric functions.
In addition, we find conditions under which the basic hypergeometric functions are in 
$q$-close-to-convex family.
\end{abstract}

\maketitle
\pagestyle{myheadings}
\markboth{Sarita Agrawal and Swadesh Sahoo}{Basic hypergeometric functions}

\section{Introduction and preliminaries}
In view of the Riemann mapping theorem, in the classical complex analysis,
the unit disk $\mathbb{D}=\{z\in\mathbb{C}:\,|z|<1\}$ is well understood
to consider as a standard domain.
The classes of convex, starlike, and close-to-convex 
functions defined in the unit disk have been extensively studied and found numerous 
applications to various problems in complex analysis and related topics. 
Part of this development is the study of subclasses of the class of 
univalent functions, more general than the classes of 
convex, starlike, and close-to-convex functions. Number of geometric characterizations of
such functions in terms of image of the unit disk are extensively studied by several authors.
Background knowledge in this theory can be found from standard books in geometric function theory
(see for instance, \cite{Dur83}) 
In this connection, our main aim is to study certain geometric properties of basic 
hypergeometric functions introduced by Heine \cite{Hei46}. 
Motivation behind this comes from mapping properties of the Gauss
hypergeometric functions studied in \cite{Kus02} in terms of convexity properties 
of shifted hypergeometric functions in the direction of the imaginary axis.
One of the key tools to study this geometric property
was the continued fraction of Gauss and a theorem of Wall concerning
a characterization of Hausdorff moment sequences by means of (continued) $g$-fractions
\cite{Wal48}. More background on mapping properties of the Gauss hypergeometric
functions can be found in \cite{HPV10,MM90,Pon97,PV01,Sil93}.

We now collect some standard notations and basic definitions used in this paper.
We denote by $\mathcal{A}$, the class of analytic functions $f(z)$ defined 
on $\mathbb{D}$ with the normalization $f(0)=0=f'(0)-1$. 
In other words, functions $f(z)$ in $\mathcal{A}$ have the power series representation
$$f(z)=z+\sum_{n=2}^\infty a_n z^n, \quad z\in\mathbb{D}.
$$
One-one analytic functions in this theory are usually called {\em univalent analytic functions}.
A function $f\in \mathcal{A}$ is called starlike ($f\in \mathcal{S}^*$) if
$${\rm Re}\,\left(\frac{zf'(z)}{f(z)}\right)>0, \quad z\in \mathbb{D}
$$
and $f\in\mathcal{A}$ is called close-to-convex ($f\in \mathcal{K}$) if there exists $g\in \mathcal{S}^*$ such that
$${\rm Re}\,\left(\frac{zf'(z)}{g(z)}\right)>0, \quad z\in \mathbb{D}.
$$
Clearly, $\mathcal{S}^*\subset \mathcal{K}$.
In 1990, a $q$-analog of starlike functions was introduced by Ismail et al. \cite{IMS90} 
via {\em the $q$-difference operator} ($q<1$), $D_qf$, defined by
the equation
\begin{equation}\label{sec1-eqn1}
(D_qf)(z)=\frac{f(z)-f(qz)}{z(1-q)},\quad z\neq 0, \quad (D_qf)(0)=f'(0).
\end{equation}

In view of the above relationship between $\mathcal{S}^*$ and $\mathcal{K}$,
with the help of the difference operator $D_qf$, a similar $q$-analog of close-to-convex 
functions are studied in \cite{RS12, SS14}. 
\begin{definition}\label{sec3-def1}
A function $f\in\mathcal{A}$ is said to belong to the class $\mathcal{K}_q$ if there exists $g\in \mathcal{S}^*$ such that
$$\left|\frac{z}{g(z)}(D_qf)(z)-\frac{1}{1-q}\right|\leq \frac{1}{1-q}, \quad z\in \mathbb{D} .
$$
\end{definition}
As $q\to 1^-$, the closed disk $|w-(1-q)^{-1}|\le (1-q)^{-1}$ reduces to the right-half 
plane ${\rm Re}\,w>0$ and hence the class $\mathcal{K}_q$ coincides with the class $\mathcal{K}$. 
We also call the function $f$ {\em the $q$-close-to-convex function}, when
$f\in \mathcal{K}_q$ with the starlike function $g$.

The difference operator $D_qf$ defined in (\ref{sec1-eqn1}) plays an important role in the theory of basic hypergeometric series and quantum physics (see for instance \cite{And74,Ern02,Fin88,Kir95,Sla66}).
It is easy to see that $D_q\to \displaystyle\frac{d}{dz}$ as $q\to 1^-$.

The well-known basic hypergeometric functions
involving {\em Watson's symbol} $(a;q)_n$ (also called {\em the $q$-shifted factorial}), $n\ge 0$,
defined by
$$(a;q)_0=1,\quad (a;q)_n=(1-a)(1-aq)(1-aq^2)\cdots (1-aq^{n-1})=\prod_{k=0}^\infty \frac{1-aq^k}{1-aq^{k+n}}
$$
for all real or complex values of $a$. The following relation is useful
in this context:
\begin{equation}\label{sec1-eqn2}
(1-a)(aq;q)_n=(a;q)_n(1-aq^n)=(a;q)_{n+1}. 
\end{equation}
In the unit disk $\mathbb{D}$, {\em Heine's hypergeometric series}
$$\sum_{n=0}^\infty \frac{(a;q)_n(b;q)_n}{(c;q)_n(q;q)_n}z^n
 = 1+\frac{(1-a)(1-b)}{(1-c)(1-q)}z+\frac{(1-a)(1-aq)(1-b)(1-bq)}{(1-c)(1-cq)(1-q)(1-q^2)}z^2+\cdots,
$$
where $|q|<1$ and $a,b,c$ are real or complex parameters, is convergent. The corresponding 
function is denoted by $\Phi[a,b;c;q,z]$ and called as the {\em basic (or Heine's) hypergeometric 
function} \cite{AAR99,Sla66}. The limit
$$\lim_{q\to 1^-}\frac{(q^a;q)_n}{(q;q)_n}=a(a+1)\cdots (a+n-1)
$$
says that, with the substitution $a\mapsto q^a$, the Heine hypergeometric function takes to the well-known 
Gauss hypergeometric function $F(a,b;c;z)$ when $q$ approaches $1^-$.

In Section~2, we show that the functions
$$
\frac{z\Phi[a,bq;cq;q,z]}{\Phi[a,b;c;q,z]} ~\left(\mbox{ or }~ \frac{z\Phi[aq,bq;cq;q,z]}{\Phi[aq,b;c;q,z]}\right),
\frac{z\Phi[aq,b;c;q,z]}{\Phi[a,b;c;q,z]} ~\left(\mbox{ or }~ \frac{z\Phi[aq,bq;cq;q,z]}{\Phi[a,bq;cq;q,z]}\right)
$$
and
$$\frac{z\Phi[aq,bq;cq;q,z]}{\Phi[a,b;c;q,z]}
$$
are analytic in a cut plane and map both the unit disk and a half-plane
univalently onto domains convex in the direction of the imaginary axis.

Section~3 deals with $q$-close-to-convexity properties of the
basic shifted hypergeometric functions $z\Phi[a,b;c;q,z]$.

%

Finally, concluding remarks on the paper have been focused in Section~4.

\section{Continued fractions and mapping properties}
In this section, we mainly concentrate on mapping properties of functions of the 
form 
$$\frac{z\Phi[aq,bq;cq;q,z]}{\Phi[a,b;c;q,z]}
~~\mbox{ or }~~
\frac{\Phi[aq,bq;cq;q,z]}{\Phi[a,b;c;q,z]}.
$$ 
First we collect few useful identities on basic hypergeometric functions.
Further, analytic properties of continued 
fraction of Gauss and Wall's characterization of Hausdorff moment sequences 
by means of (continued) $g$-fractions \cite{Wal48} are used as important tools, 
and finally, the following lemma has been used to conclude the results.

\begin{lemma}\cite{Kus02,Mer59}
Let $\mu:\,[0,1]\to [0,1]$ be non-decreasing with $\mu(1)-\mu(0)=1$. Then the function
$$z\mapsto \int_0^1\frac{z}{1-tz}\,d\mu(t)
$$
is analytic in the cut-plane $\mathbb{C}\setminus [1,\infty]$ and maps both the unit disk and
the half-plane $\{z\in \mathbb{C}:\,{\rm Re}\,z <1\}$ univalently onto domains convex
in the direction of the imaginary axis.
\label{sec2-lem0}
\end{lemma}
Here, a domain $D\subset\mathbb{C}$ is called {\em convex in the direction of the 
imaginary axis} \cite{Rob36,RZ76} if the intersection of $D$ with any line parallel to the
imaginary axis is either empty or a line segment. As an application of Lemma~\ref{sec2-lem0},
subject to some ranges for the real parameters $a,b,c$, it is proved in \cite{Kus02} that
the hypergeometric function $z\mapsto F(a,b;c;z)$ as well as the shifted function 
$z\mapsto zF(a,b;c;z)$ each maps both the unit disk $\mathbb{D}$ and 
the half-plane $\{z\in\mathbb{C}:\,{\rm Re}\,z<1\}$
univalently onto domains convex in the direction of the imaginary axis. 
\begin{figure}[H]
\includegraphics[width=6cm]{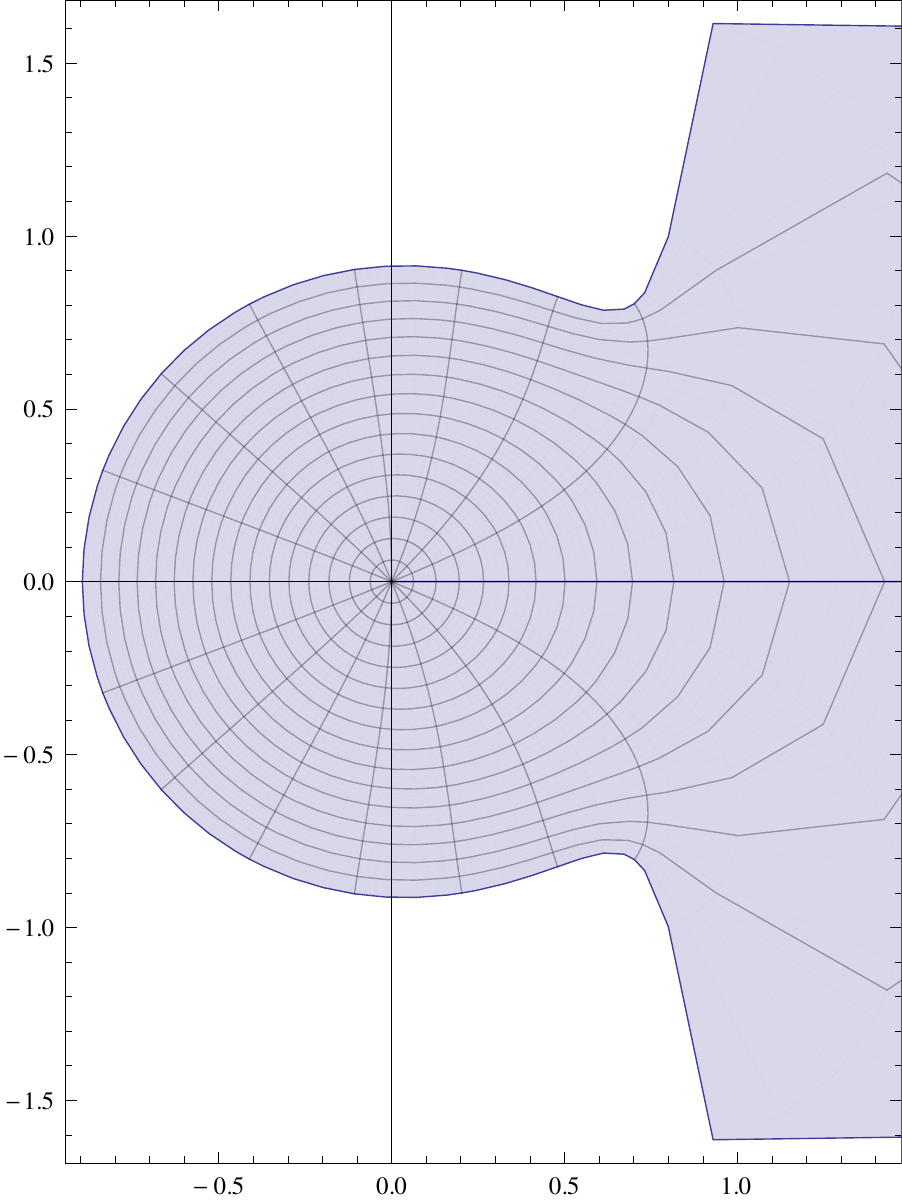}
\caption{The image of the disk $|z|<r$ ($r=0.999$) under the mapping $zF(a+1,b;c;z)/F(a,b;c;z)$, 
when $a=0$, $b=0.0199$, $c=0.1$.}\label{ParametricPlotH1}
\end{figure}
Moreover, he obtained similar properties of images under ratios of
hypergeometric functions having shifted parameters. For instance, 
see Figure~\ref{ParametricPlotH1} for description of such a function.
In order to use analytic properties of continued fraction of Gauss,
certain identities on the Gauss hypergeometric functions were crucial to
consider. In this context, it is also important to collect similar relations
on basic hypergeometric functions. One such relation is obtained in  \cite{IMS90}
and we also use that relation in our proofs.

\begin{lemma}\label{sec2-lem1}
The basic hypergeometric function of Heine $\Phi[a,b;c;q,z]$ is satisfied by the identities
\begin{enumerate}
\item[{\bf (a)}] $\Phi[a,b;c;q,z]-\Phi[a,bq;cq;q,z]=\displaystyle\frac{(1-a)(c-b)}{(1-c)(1-cq)}z\Phi[aq,bq;cq^2;q,z]$;
\item[{\bf (b)}] 
\begin{eqnarray*}
\Phi[aq,b;c;q,z]-\Phi[a,b;c;q,z] &=& \displaystyle\frac{a(1-b)}{(1-c)}z\Phi[aq,bq;cq;q,z]\\
&=& \frac{a}{1-a}(\Phi[a,b;c;q,z]-\Phi[a,b;c;q,qz]) .
\end{eqnarray*}
\end{enumerate}
\end{lemma}
\begin{proof}
\begin{enumerate}
\item[{\bf (a)}] Making use of the identities given in (\ref{sec1-eqn2}), we have
\begin{eqnarray*}
\Phi[a,b;c;q,z]-\Phi[a,bq;cq;q,z]
&=& \sum_{n=0}^\infty \frac{(a;q)_n(b;q)_n}{(c;q)_n(q;q)_n}z^n -\sum_{n=0}^\infty \frac{(a;q)_n(bq;q)_n}{(cq;q)_n(q;q)_n}z^n\\
&=& \sum_{n=0}^\infty \frac{(a;q)_n(b;q)_n}{(c;q)_n(q;q)_n}\left[1-\frac{(1-bq^n)(1-c)}{(1-cq^n)(1-b)}\right]z^n\\
&=&\sum_{n=0}^\infty \frac{(a;q)_n(b;q)_n}{(c;q)_n(q;q)_n}\left[\frac{(c-b)(1-q^n)}{(1-b)(1-cq^n)}\right]z^n.
\end{eqnarray*}
Since the first term (when $n=0$) vanishes in the above sum, by rewriting the summation,
we get
\begin{eqnarray*}
\Phi[a,b;c;q,z]-\Phi[a,bq;cq;q,z]
&=&\sum_{n=0}^\infty \frac{(a;q)_{n+1}(b;q)_{n+1}}{(c;q)_{n+1}(q;q)_{n+1}}
\left[\frac{(c-b)(1-q^{n+1})}{(1-b)(1-cq^{n+1})}\right]z^{n+1}\\
&=&\sum_{n=0}^\infty \frac{(1-a)(aq;q)_n(1-b)(bq;q)_n(c-b)}
{(1-c)(1-cq)(cq^2;q)_n(q;q)_n(1-b)}z^{n+1}\\
&=&\frac{(1-a)(c-b)}{(1-c)(1-cq)}z\,\Phi[aq,bq;cq^2;q,z].
\end{eqnarray*}
\item[{\bf (b)}] By similar steps used in the proof of {\bf (a)}, we obtain
\begin{eqnarray*}
\Phi[aq,b;c;q,z]-\Phi[a,b;c;q,z]
&=&\sum_{n=0}^\infty \frac{(aq;q)_n(b;q)_n}{(c;q)_n(q;q)_n}
\left[1-\frac{1-a}{1-aq^n}\right] z^n\\
&=&\sum_{n=0}^\infty \frac{(aq;q)_{n+1}(b;q)_{n+1}}{(c;q)_{n+1}(q;q)_{n+1}}\frac{a(1-q^{n+1})}{(1-aq^{n+1})} z^{n+1}.
\end{eqnarray*}
Now, we use the relation (\ref{sec1-eqn2}) and obtain the difference
\begin{eqnarray*}
\Phi[aq,b;c;q,z]-\Phi[a,b;c;q,z]
&=&\sum_{n=0}^\infty \frac{a(aq;q)_n(1-b)(bq;q)_n}
{(1-c)(cq;q)_n(q;q)_n}z^{n+1}\\
&=&\frac{a(1-b)}{(1-c)}z\,\Phi[aq,bq;cq;q,z].
\end{eqnarray*}
Finally, the identity 
$$\frac{a}{1-a}\left(\Phi[a,b;c;q,z]-\Phi[a,b;c;q,qz]\right)=\frac{a(1-b)}{(1-c)}z \Phi[aq,bq;cq;q,z]
$$
follows from a similar identity obtained in \cite{IMS90}.
\end{enumerate}
\end{proof}

The following subsections deal with mapping properties discussed above. In particular, we generalize certain results of K\"{u}stner \cite{Kus02}.
\subsection{The ratio $\displaystyle \frac{z\Phi[a,bq;cq;q,z]}{\Phi[a,b;c;q,z]}$ or 
$\displaystyle\frac{z\Phi[aq,bq;cq;q,z]}{\Phi[aq,b;c;q,z]}$}

Figure~\ref{ParametricPlotQ1} visualizes the behaviour of the image domain of the disk $|z|<0.998$ under
the map $z\Phi[a,bq;cq;q,z]/\Phi[a,b;c;q,z]$ when $a=0.9$, $b=0.7$, $c=0.6$, $q=0.8$.
\begin{figure}[H]
\includegraphics[width=9cm]{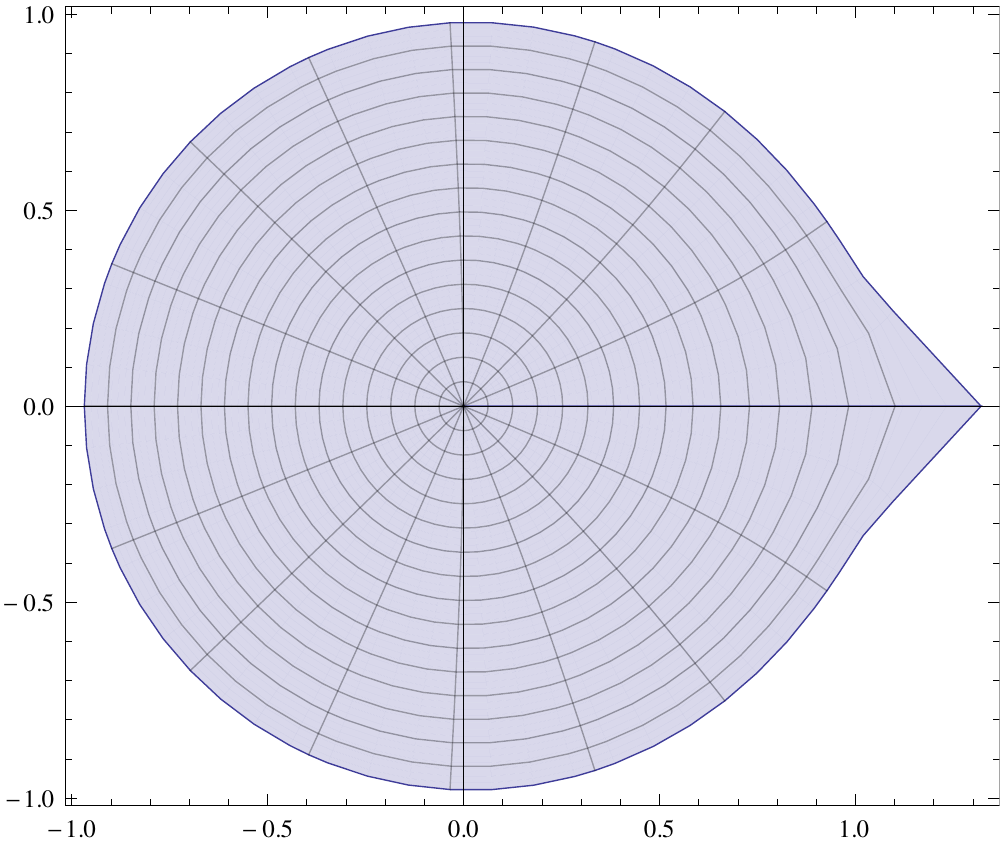}
\caption{The image of the disk $|z|<0.998$ under the mapping $z\Phi[a,bq;cq;q,z]/\Phi[a,b;c;q,z]$, 
when $a=0.9$, $b=0.7$, $c=0.6$, $q=0.8$.}\label{ParametricPlotQ1} 
\end{figure}
This shows that the map $z\Phi[a,bq;cq;q,z]/\Phi[a,b;c;q,z]$ in general does not take unit disk onto convex domains in all the directions. The following result obtains conditions on the parameters
$a,b,c$ for which the image domain is convex in the direction of the imaginary axis.
\begin{theorem}\label{sec2-thm1}
For $q \in (0,1)$ suppose that $a,b,c$ be non-negative real numbers satisfying $0\le q(b-c)\le 1-cq$ and $0 < a-c\le 1-c$.
Then there exists a non-decreasing function $\mu:\,[0,1]\to[0,1]$ with $\mu(1)-\mu(0)=1$ such that 
$$\frac{z\Phi[a,bq;cq;q,qz]}{\Phi[a,b;c;q,qz]}=\int_0^1 \frac{z}{1-tz} \mbox{d} \mu(t)
$$
which is analytic in the cut-plane $\mathbb{C}\setminus [1,\infty]$ and maps both the unit disk and
the half-plane $\{z\in \mathbb{C}:\,{\rm Re}\,z <1\}$ univalently onto domains convex
in the direction of the imaginary axis.
\end{theorem}
\begin{proof}
First of all we find the continued fraction of the ratio $z\phi_1/\phi_0$,
where $\phi_1 =\Phi[a,bq;cq;q,z]$ and $\phi_0 =\Phi[a,b;c;q,z]$.
Consider the iteration
\begin{equation}\label{thm2-eq1}\phi_{i-1}-\phi_i=d_i z \phi_{i+1},\quad \mbox{$i=1,2,3,\ldots$}
\end{equation}
where $d_i$'s are to be computed for each $i$. Rewrite this iteration in the form
\begin{equation}\label{thm2-eq2}
\frac{\phi_i}{\phi_{i-1}}=\frac{1}{1+d_i z \displaystyle\frac{\phi_{i+1}}{\phi i}}, \quad \mbox{$i=1,2,3,\ldots$}.
\end{equation}
Starting with $i=1$, the relation (\ref{thm2-eq2}) yields the following continued fraction for 
$\phi_1/\phi_0$:
$$\frac{\phi_1}{\phi_0} = \frac{1}{1+d_1 z\displaystyle\frac{\phi_2}{\phi_1}}
= \frac{1}{1+}\,\frac{d_1 z}{1+}\,\,d_2 z\frac{\phi_3}{\phi_2}
= \frac{1}{1+}\,\,\frac{d_1 z}{1+}\,\,\frac{d_2 z}{1+}\,\,d_3 z \frac{\phi_4}{\phi_3}.
$$
Continuing in this manner, it leads to the continued fraction
\begin{equation}\label{thm2-eq3}
\frac{\phi_1}{\phi_0}=\frac{\Phi[a,bq;cq;q,z]}{\Phi[a,b;c;q,z]}
=\frac{1}{1+}\,\,\frac{d_1 z}{1+}\,\,\frac{d_2 z}{1+}\,\,\frac{d_3 z}{1+ \ldots}.
\end{equation}
We now calculate the values of $d_i$ for all $i$. First, to find $d_1$, we use Lemma~\ref{sec2-lem1}{\bf (a)}
and see that 
$$\phi_0 -\phi_1 = \Phi[a,b;c;q,z]-\Phi[a,bq;cq;q,z]=\frac{(1-a)(c-b)}{(1-c)(1-cq)}z\,\Phi[aq,bq;cq^2;q,z].
$$
Comparing with (\ref{thm2-eq1}), for $i=1$, we get
$$d_1=\frac{(1-a)(c-b)}{(1-c)(1-cq)} ~~\mbox{ and }~~ \phi_2=\Phi[aq,bq;cq^2;q,z].
$$
A similar computation as in Lemma~\ref{sec2-lem1}{\bf (a)} gives
$$\phi_1 -\phi_2=\Phi[a,bq;cq;q,z]-\Phi[aq,bq;cq^2;q,z]=\frac{(1-bq)(cq-a)}{(1-cq)(1-cq^2)}z\,\Phi[aq,bq^2;cq^3;q,z].
$$
Again by comparing with (\ref{thm2-eq1}), for $i=2$, we get
$$d_2=\frac{(1-bq)(cq-a)}{(1-cq)(1-cq^2)}~~\mbox{ and }~~ \phi_3=\Phi[aq,bq^2;cq^3;q,z].
$$
By a similar technique one can compute
$$d_3= q\,\frac{(1-aq)(cq-b)}{(1-cq^2)(1-cq^3)}~~\mbox{ and }~~ d_4=q\,\frac{(1-bq^2)(cq^2-a)}{(1-cq^3)(1-cq^4)}.
$$
Therefore, inductively we obtain
$$d_{2n+1}= q^n\,\frac{(1-aq^n)(cq^n-b)}{(1-cq^{2n})(1-cq^{2n+1})},\quad \mbox{for $n\geq 0$}
$$
and 
$$d_{2n}= q^{n-1}\,\frac{(1-bq^n)(cq^n-a)}{(1-cq^{2n-1})(1-cq^{2n})}, \quad \mbox{for $n\geq 1$}.
$$
In order to apply the notion of the Hausdorff moment sequences by means of
(continued) $g$-fractions, a technique used in \cite{Kus02}, 
we first rewrite (\ref{thm2-eq3}) in the form
$$\frac{\Phi[a,bq;cq;q,z]}{\Phi[a,b;c;q,z]}=
\frac{1}{1-}\,\,\frac{b_1 z}{1-}\,\,\frac{b_2 z}{1-}\,\,\frac{b_3 z}{1- \ldots}.
$$
Then we get
$$b_{2n+1}= q^n\,\frac{(1-aq^n)(b-cq^n)}{(1-cq^{2n})(1-cq^{2n+1})},\quad \mbox{for $n\geq 0$}
$$
and
$$b_{2n}= q^{n-1}\,\frac{(1-bq^n)(a-cq^n)}{(1-cq^{2n-1})(1-cq^{2n})}, \quad \mbox{for $n\geq 1$}.
$$
Now by replacing $z$ by $qz$, we have
$$\frac{\Phi[a,bq;cq;q,qz]}{\Phi[a,b;c;q,qz]}
=\frac{1}{1-}\,\,\frac{b_1 qz}{1-}\,\,\frac{b_2 qz}{1-}\,\,\frac{b_3 qz}{1- \ldots}
=\frac{1}{1-}\,\,\frac{a_1 z}{1-}\,\,\frac{a_2 z}{1-}\,\,\frac{a_3 z}{1- \ldots}
$$
where $a_i=b_i q$ with
$$a_{2n+1}= q^{n+1}\,\frac{(1-aq^n)(b-cq^n)}{(1-cq^{2n})(1-cq^{2n+1})}\quad\mbox{ for $n\geq0$},
$$
and 
$$a_{2n}= q^n \,\frac{(1-bq^n)(a-cq^n)}{(1-cq^{2n-1})(1-cq^{2n})}\quad \mbox{ for $n\geq 1$}.
$$
Set $a_i=(1-g_i)g_{i+1}$ for each $i$. Then, the ratio 
$\Phi[a,bq;cq;q,qz]/ \Phi[a,b;c;q,qz]$ has the continued fraction (also called a $g$-fraction) 
$$\frac{\Phi[a,bq;cq;q,qz]}{\Phi[a,b;c;q,qz]}
=\frac{1}{1-}\,\,\frac{(1-g_1)g_2 z}{1-}\,\,\frac{(1-g_2)g_3 z}{1-}\,\,\frac{(1-g_3)g_4 z}{1- \ldots}
$$
in terms of the moment sequence $<g_i>$ given by
$$g_{2n+1}=q^n\left(\frac{a-cq^n}{1-cq^{2n}}\right),\quad \mbox{ $n\geq 0$}
$$
and
$$g_{2n}=q^n\left(\frac{b-cq^{n-1}}{1-cq^{2n-1}}\right),\quad \mbox{ $n\geq 1$}.
$$
Note that the moment sequence $<g_i>$ should satisfy the relation $0\leq g_i\leq 1$,
when we apply Wall's theorem \cite{Wal48}. 
By hypothesis, it is clear that $0\le g_1,g_2\le 1$. Using this, it is now easy to 
verify the relation $0\leq g_i\leq 1$ for all $i$. Indeed, 
since $b\ge c>cq^{n-1}$ and 
$1\ge cq>cq^{2n-1}$, we get the lower bound for $g_i$. 
Next,
as $bq-cq<1-cq$, we have $bq<1$ and hence $bq^n<1$ which implies
$bq^n-cq^{2n-1}<1- cq^{2n-1}$. 
Other required conditions can be proved similarly.
Hence, there exists a 
non-decreasing function $\mu:\,[0,1]\to[0,1]$ such that $\mu(1)-\mu(0)=1$ and 
\begin{equation}\label{eq}
\frac{\Phi[a,bq;cq;q,qz]}{\Phi[a,b;c;q,qz]}=\int_0^1 \frac{1}{1-tz} \mbox{d} \mu(t).
\end{equation}
This concludes the proof of our theorem.
\end{proof}

\begin{corollary}\label{cor1}
For $q \in (0,1)$ suppose that $a,b,c$ be non-negative real numbers satisfying $0\le q(b-c)\le 1-cq$ and $0 < a-c\le 1-c$.
Then there exists a non-decreasing function $\mu:\,[0,1]\to[0,1]$ with $\mu(1)-\mu(0)=1$ such that 
$$\frac{z\Phi[a,bq;cq;q,z]}{\Phi[a,b;c;q,z]}=\int_0^1 \frac{qz}{q-tz} \mbox{d} \mu(t)
$$
which is analytic in the cut-plane $\mathbb{C}\setminus [q,\infty]$ and maps both the unit disk and
the half-plane $\{z\in \mathbb{C}:\,{\rm Re}\,z <q\}$ univalently onto domains convex
in the direction of the imaginary axis.
\end{corollary}
\begin{proof}
Replacing $z$ by $z/q$ in (\ref{eq}), we have
$$\frac{\Phi[a,bq;cq;q,z]}{\Phi[a,b;c;q,z]}=\int_0^1 \frac{q}{q-tz} \mbox{d} \mu(t).
$$
Thus, the assertion of our corollary follows.
\end{proof}
\begin{remark}
If we substitute $a$ by $aq$ in (\ref{eq}), we get the same integral expression
for the ratio $z\Phi[aq,bq;cq;q,z]/{\Phi[aq,b;c;q,z]}$. Moreover, if we substitute $a$ by $q^a$,
$b$ by $q^b$ and $c$ by $q^c$, we obtain a result of K\"ustner (see \cite[Theorem~1.5]{Kus02})
in the limiting sense when $q\to 1^{-}$.
\end{remark}

\subsection{The ratio $\displaystyle \frac{z\Phi[aq,b;c;q,z]}{\Phi[a,b;c;q,z]}$ 
or $\displaystyle \frac{z\Phi[aq,bq;cq;q,z]}{\Phi[a,bq;cq;q,z]}$}

Figure~\ref{ParametricPlotQ2} visualizes the behaviour of the image domain of the disk $|z|<0.999$ under
the map $z\Phi[aq,b;c;q,z]/\Phi[a,b;c;q,z]$ when $a=0.99$, $b=0.998$, $c=0.98$, $q=0.9$.
\begin{figure}[H]
\includegraphics[width=8cm]{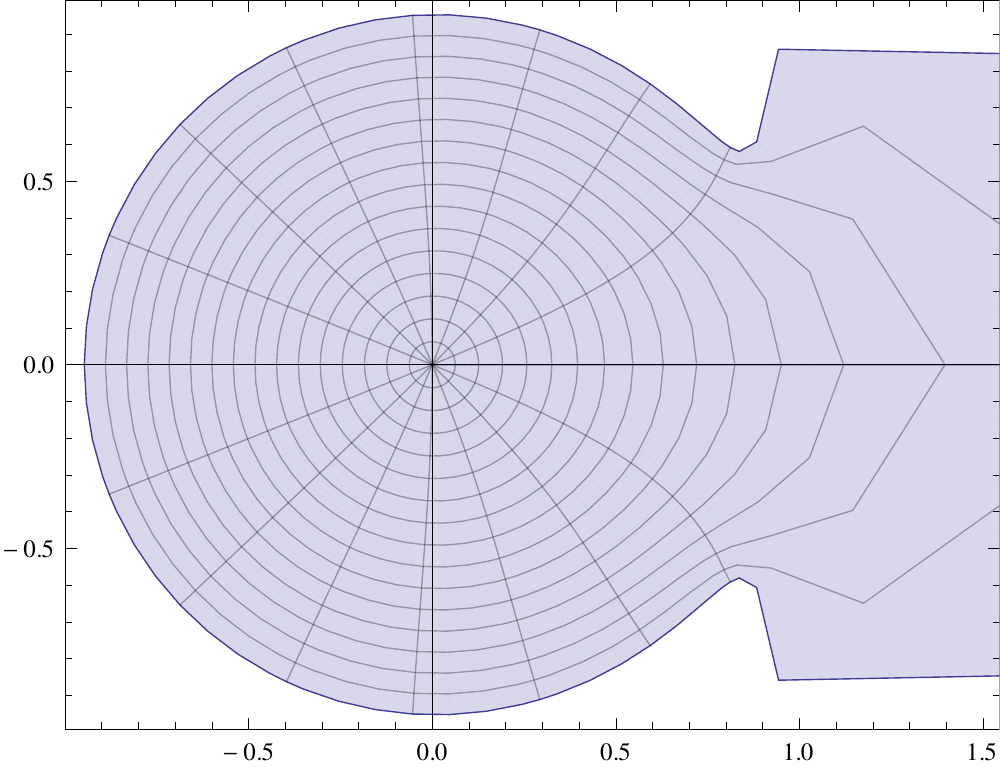}
\caption{The image of the disk $|z|<0.999$ under the mapping $\Phi[aq,b;c;q,z]/\Phi[a,b;c;q,z]$, 
when $a=0.99$, $b=0.998$, $c=0.98$, $q=0.9$.}\label{ParametricPlotQ2} 
\end{figure}
This shows that the map $z\Phi[aq,b;c;q,z]/\Phi[a,b;c;q,z]$ in general does not take 
the unit disk onto domains convex in all the directions. 
The following result obtains conditions on the parameters
$a,b,c$ for which the image domain is convex in the direction of the imaginary axis.

\begin{theorem}\label{sec2-thm2}
For $q \in (0,1)$ suppose that $a,b,c$ be non-negative real numbers satisfying $0\le 1-aq\le 1-cq$ and $0 < 1-b\le 1-c$. 
Then there exists a non-decreasing function $\mu:\,[0,1]\to[0,1]$ with $\mu(1)-\mu(0)=1$ such that 
$$\frac{z\Phi[aq,b;c;q,z]}{\Phi[a,b;c;q,z]}=\int_0^1 \frac{z}{1-tz} \mbox{d} \mu(t)
$$
which is analytic in the cut-plane $\mathbb{C}\setminus [1,\infty]$ and maps both the unit disk and
the half-plane $\{z\in \mathbb{C}:\,{\rm Re}\,z <1\}$ univalently onto domains convex
in the direction of the imaginary axis.
\end{theorem}
\begin{proof}
In order to find the continued fraction of the ratio ${\Phi[aq,b;c;q,z]}/{\Phi[a,b;c;q,z]}$,
let us first consider the continued fraction of the ratio ${\Phi[a,bq;cq;q,z]}/{\Phi[a,b;c;q,z]}$
obtained in the proof of Theorem~\ref{sec2-thm1}. Now, by replacing $a$ by $aq$, we get the 
continued fraction of ${\Phi[aq,bq;cq;q,z]}/{\Phi[aq,b;c;q,z]}$, say,
$$\frac{\Phi[aq,bq;cq;q,z]}{\Phi[aq,b;c;q,z]}=
\frac{1}{1-}\,\,\frac{c_1 z}{1-}\,\,\frac{c_2 z}{1-}\,\,\frac{c_3 z}{1- \ldots}
$$
where
$$ c_{2n+1}= q^n\,\frac{(1-aq^{n+1})(b-cq^n)}{(1-cq^{2n})(1-cq^{2n+1})}, \quad \mbox{ for $n\geq0$}
$$
and 
$$c_{2n}= q^n \,\frac{(1-bq^n)(a-cq^{n-1})}{(1-cq^{2n-1})(1-cq^{2n})}\quad \mbox{ for $n\geq 1$}.
$$
Now, by Lemma~\ref{sec2-lem1}{\bf (b)}, we have
$$\Phi[aq,b;c;q,z]-\Phi[a,b;c;q,z]=\frac{a(1-b)}{(1-c)}z\,\Phi[aq,bq;cq;q,z].
$$
Simplifying this, we get
$$\frac{\Phi[a,b;c;q,z]}{\Phi[aq,b;c;q,z]}
= 1-\frac{a(1-b)}{(1-c)}z\,\frac{\Phi[aq,bq;cq;q,z]}{\Phi[aq,b;c;q,z]}.
$$
This implies
$$\frac{\Phi[aq,b;c;q,z]}{\Phi[a,b;c;q,z]}
=\frac{1}{1-\displaystyle\frac{a(1-b)}{(1-c)}z\,\frac{\Phi[aq,bq;cq;q,z]}{\Phi[aq,b;c;q,z]}}
=\frac{1}{1-}\frac{\displaystyle\frac{a(1-b)z}{(1-c)}}{1-}\,\,\frac{c_1 z}{1-}
\quad\frac{c_2 z}{1-}\,\,\frac{c_3 z}{1- \ldots},
$$
where $c_i$'s are defined as above.
Rewriting this continued fraction by means of continued $g$-fractions of the form
$$\frac{\Phi[aq,b;c;q,z]}{\Phi[a,b;c;q,z]}=\frac{1}{1-}\,\,\,\frac{(1-g_0)g_1 z}{1-}
\frac{(1-g_1)g_2 z}{1-}\,\,\frac{(1-g_2)g_3 z}{1-\ldots},
$$
we get
$$g_{2n}=\frac{1-aq^n}{1-cq^{2n-1}}\, \quad \mbox{ for $n \geq 1$}
$$
and
$$g_{2n+1}=\frac{1-bq^n}{1-cq^{2n}}\, \quad \mbox{ for $n \geq 0$}
$$
with 
$g_0=1-a$.
By a similar technique as in the proof of Theorem~\ref{sec2-thm1}, one can show by 
using the hypothesis that
$0 \leq g_i \leq 1$ for all $i$.
Hence, Wall's theorem shows that there exists a non-decreasing function $\mu:\,[0,1]\to[0,1]$ such that $\mu(1)-\mu(0)=1$ and 
\begin{equation}\label{thm4-eq2}
\frac{\Phi[aq,b;c;q,z]}{\Phi[a,b;c;q,z]}=\int_0^1 \frac{1}{1-tz} \mbox{d} \mu(t).
\end{equation}
Thus, the assertion of our theorem follows.
\end{proof}

\begin{remark}
If we substitute $b$ by $bq$ and $c$ by $cq$ in (\ref{thm4-eq2}), we get the same integral expression
for the ratio $z\Phi[aq,bq;cq;q,z]/{\Phi[a,bq;cq;q,z]}$. Moreover, if we substitute $a$ by $q^a$,
$b$ by $q^b$ and $c$ by $q^c$, and apply the limit as $q\to 1^{-}$, we obtain a 
result of K\"ustner (see \cite[Theorem~1.5]{Kus02}).
\end{remark}

\subsection{The Ratio $\displaystyle \frac{z\Phi[aq,bq;cq;q,z]}{\Phi[a,b;c;q,z]}$}

Figure~\ref{ParametricPlotQ3} visualizes the behaviour of the image domain of the disk $|z|<0.999$ under
the map $z\Phi[aq,bq;cq;q,z]/\Phi[a,b;c;q,z]$ when $a=0.99$, $b=0.998$, $c=0.98$, $q=0.9$.
\begin{figure}[H]
\includegraphics[width=5cm]{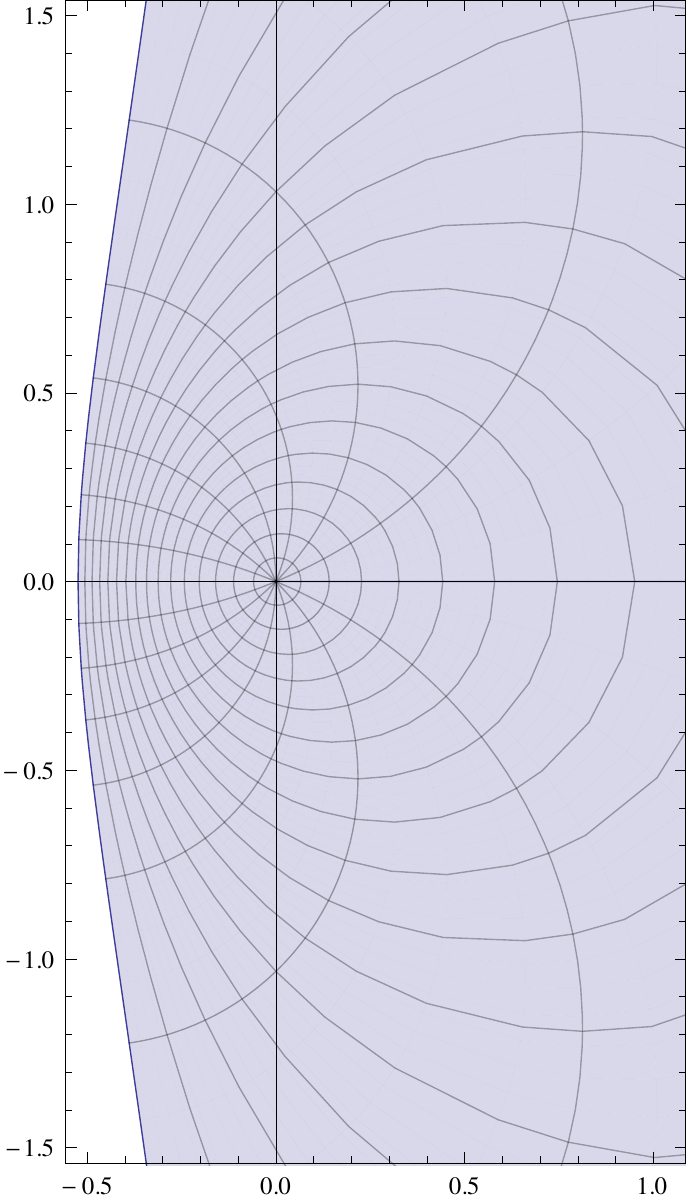}
\caption{The image of the disk $|z|<0.999$ under the mapping $\Phi[aq,bq;cq;q,z]/\Phi[a,b;c;q,z]$, 
when $a=0.99$, $b=0.998$, $c=0.98$, $q=0.9$.}\label{ParametricPlotQ3} 
\end{figure}
The following result obtains conditions on the parameters
$a,b,c$ for which the image domain will be convex in the direction of the imaginary axis.
\begin{theorem}\label{sec2-thm3}
For $q \in (0,1)$ suppose that $a,b,c$ be non-negative real numbers satisfying $0\le 1-aq\le 1-cq$ and $0 < 1-b\le 1-c$.
Then there exists a non-decreasing function $\mu:\,[0,1]\to[0,1]$ with $\mu(1)-\mu(0)=1$ such that 
$$\frac{z\Phi[aq,bq;cq;q,z]}{\Phi[a,b;c;q,z]}=\frac{1}{a}\int_0^1 \frac{z}{1-tz} \mbox{d} \mu(t)
$$
which is analytic in the cut-plane $\mathbb{C}\setminus [1,\infty]$ and maps both the unit disk and
the half-plane $\{z\in \mathbb{C}:\,{\rm Re}\,z <1\}$ univalently onto domains convex
in the direction of the imaginary axis.
\end{theorem}
\begin{proof}
From the difference equation of Lemma~\ref{sec2-lem1}{\bf (b)} and Theorem~\ref{sec2-thm2}, we have 
$$\displaystyle \frac{z\Phi[aq,bq;cq;q,z]}{\Phi[a,b;c;q,z]}=\frac{(1-c)}{a(1-b)}
\left[\displaystyle \frac{\Phi[aq,b;c;q,z]}{\Phi[a,b;c;q,z]} -1 \right]
=\frac{(1-c)}{a(1-b)}\left[\int_0^1\frac{1}{1-tz}\,d\mu_0(t)-1\right],
$$
for some non-decreasing function $\mu_0:\,[0,1]\to [0,1]$ with $\mu_0(1)-\mu_0(0)=1$.
Define
$$\mu_1(t):=\frac{1}{g_1}\int_0^t s\, d\mu_0(s)
$$
for $g_1=(1-b)/(1-c)>0$ as in the proof of Theorem~\ref{sec2-thm2}. It follows from \cite[Remark~3.2]{Kus02} that 
$$\frac{\Phi[aq,b;c;q,z]}{\Phi[a,b;c;q,z]}=\int_0^1 \frac{1}{1-tz}\, d\mu_0(t)
=1+ g_1 \int_0^1 \frac{z}{1-tz}\, d\mu_1(t)
$$
where $\mu_1$ is also a non-decreasing self-mapping of $[0,1]$ with $\mu_1(1)-\mu_1(0)=1$.
Finally, we get
$$ \frac{z\Phi[aq,bq;cq;q,z]}{\Phi[a,b;c;q,z]}=\frac{(1-c)}{a(1-b)}\,\,g_1 \int_0^1 \frac{z}{1-tz}\, d\mu_1(t)
= \frac{1}{a}\int_0^1 \frac{z}{1-tz}\,d\mu_1(t)
$$
and thus, Lemma~\ref{sec2-lem0} proves the conclusion of our theorem.
\end{proof}
\begin{remark}
If we substitute $a$ by $q^a$, $b$ by $q^b$ and $c$ by $q^c$, then as $q\to 1^{-}$, 
we get the result of K\"ustner \cite[Theorem~1.5]{Kus02} for the ratio 
$zF[a+1,b+1;c+1;z]/{F[a,b;c;z]}$ of the Gauss hypergeometric
functions. This function has also the similar mapping properties.
\end{remark}

\section{The $q$-close-to-convexity property}

The $q$-close-to-convex functions (see Definition~\ref{sec3-def1}), defined in Section~1,
analytically characterizes by the fact that
$|g(z)+f(qz)-f(z)|/|g(z)|\le 1$ for all $z\in\mathbb{D}$ (see \cite[Lemma~3.1]{SS14}).
It shows that if the function $g(z)$ vanishes at $z$ then $z$ has to be zero, else 
the quotient $(g(z)+f(qz)-f(z))/g(z)$ would have a pole at $z=0$.

We recall the following lemma from \cite{PV01} concerning a sufficient condition for 
the shifted Gauss hypergeometric functions $zF(a,b;c;z)$ to be in $\mathcal{K}$.

\begin{lemma}\cite[Theorem~2.1]{PV01}
Define $T_1(a,b):=\max\{a+b,a+b+(ab-1)/2,2ab\}$ for $a,b>0$. Suppose that $c$ satisfies either
$c\geq T_1(a,b)$ or $c=a+b$ with
$$ab\geq 1, \quad a+b\leq 2ab ~\mbox{ and }~\frac{\Gamma(a+b)}{\Gamma(a)\Gamma(b)}\leq 2 .
$$
Then $zF(a,b;c;z)$ is close-to-convex with $g(z)=z/(1-z)$.
\label{sec3-lem1}
\end{lemma}
Number of problems on the convexity, starlikeness, and close-to-convexity properties of 
the Gauss hypergeometric functions are investigated in \cite{HPV10,Pon97,PV01,Sil93}.
In fact, a large number of open problems on the starlikeness of hypergeometric functions
are remained unsolved. 
Our objective in this section is to extend Lemma~\ref{sec3-lem1} associated with the shifted 
basic hypergeometric function $z\Phi[a,b;c;q,z]$. The following theorem in this direction 
improves a result obtained in \cite{RS12}. 
\begin{theorem}\label{thm2}
If $a, b <1$,
\begin{eqnarray*} 
T_1(a, b) &=& \min\left\{ab, ab+\frac{aq+bq-q-2ab+ab/q}{2(1-q)}, ab+\frac{aq+bq-q-2ab+ab/q}{(1-q)}\right.\\
&&\hspace*{9.5cm}\left.+\frac{a+b-q-ab/q}{(1-q)}\right\}
\end{eqnarray*}
and $c$ satisfies either
\begin{equation}\label{eqn1}
 c\le T_1(a,b)
\end{equation}
or $c=ab$ with
\begin{equation}\label{eqn2}
ab\ge \frac{aq+bq-q}{2-1/q},aq+bq+a+b-2q\le 2ab\mbox{ and }\frac{\Gamma_q(\log_q ab)}{\Gamma_q(\log_q a)\Gamma_q(\log_q b)}\le 2
\end{equation}
then $z\Phi[a,b;c;q,z]\in \mathcal{K}_q$ with the starlike function $g(z)=z/(1-z)$.
\end{theorem}
For its proof we use the following result, a generalization of a result 
by MacGregor \cite[Theorem~1]{Mac69}, recently obtained in \cite{SS14}.
\begin{lemma}\cite{SS14}\label{lem1}
Let $\{A_n\}$ be a sequence of real numbers such that $A_1=1$ and for all $n\ge 1$, define $B_n=A_n(1-q^n)/(1-q)$. Suppose that
$$1\geq B_2 \geq \cdots \geq B_n \geq \cdots \geq 0 ,
$$
 or,
$$1\leq B_2 \leq \cdots \leq B_n \leq \cdots \leq 2
$$
holds. Then $f(z)=z+ \displaystyle\sum_{n=2}^\infty A_nz^n \in\mathcal{K}_q$ with $g(z)=z/(1-z)$.
\end{lemma}
The following limit formula is also used in the proof of Theorem~\ref{thm2}.
\begin{lemma}\label{lem2}
For $0<q<1$, we have
$$\lim_{n \to \infty} \frac {(q^a,q)_n(q^b,q)_n}{(q^c,q)_n(q,q)_n}=(1-q)^{c-a-b+1} 
\frac {\Gamma_q (c)}{\Gamma_q (a) \Gamma_q (b)}.
$$
\end{lemma}
\begin{proof}
It suffices to show
$$\frac{\Gamma_q (a) \Gamma_q (b)}{\Gamma_q (c)}\lim_{n \to \infty} 
\frac {(q^a,q)_n(q^b,q)_n}{(q^c,q)_n(q,q)_n}=(1-q)^{c-a-b+1}.
$$
Now,
\begin{eqnarray*}
 && \frac{\Gamma_q (a) \Gamma_q (b)}{\Gamma_q (c)}\lim_{n \to \infty} 
\frac {(q^a,q)_n(q^b,q)_n}{(q^c,q)_n(q,q)_n}\\
&=&\lim_{n \to \infty}
\frac{(q,q)_n (1-q)^{1-a}(1-q^{n+1})^a (q,q)_n (1-q)^{1-b}(1-q^{n+1})^b (q^c,q)_{n+1} (q^a,q)_n(q^b,q)_n}
{(q^a,q)_{n+1}(q^b,q)_{n+1}(q,q)_n (1-q)^{1-c}(1-q^{n+1})^c (q^c,q)_n(q,q)_n}\\
&=&\lim_{n \to \infty}\frac{(1-q)^{c-a-b+1}(q^c,q)_{n+1}(q^a,q)_n(q^b,q)_n}
{(1-q^{n+1})^{c-a-b}(q^c,q)_n(q^a,q)_{n+1}(q^b,q)_{n+1}}\\
&=&\lim_{n \to \infty}\frac{(1-q)^{c-a-b+1}(1-q^{n+c})}
{(1-q^{n+1})^{c-a-b}(1-q^{n+a})(1-q^{n+b})}\\
&=& (1-q)^{c-a-b+1}. 
\end{eqnarray*}
This completes the proof of our lemma.
\end{proof}

\begin{remark}\label{rem1}
Taking $q \to 1^-$, the limit expression in Lemma \ref{lem2}, coincides with the well known fact
$$
\lim_{n \to \infty} \frac{(a)_n (b)_n}{(c)_n n!} =\left \{ \begin{array}{ll}
\displaystyle \frac{\Gamma(c)}{\Gamma(a)\Gamma(b)} & \mbox{ for } c+1=a+b\\
0 & \mbox{ for } c+1>a+b\\
\infty & \mbox{ for } c+1<a+b\\
\end{array}
\right.
$$
described in \cite{PV01}.
\end{remark}
We next use the limiting value 
$$\lim_{n \to \infty} \frac {(a,q)_n(b,q)_n}{(c,q)_n(q,q)_n}
=(1-q)^{\displaystyle \log_q c-\log_q a-\log_q b+1} \frac {\Gamma_q (\log_q c)}{\Gamma_q 
(\log_q a) \Gamma_q (\log_q b)}
$$
which can be easily verified with the substitutions $q^a\to a$, $q^b\to b$ and $q^c\to c$.
\subsection{Proof of Theorem \ref{thm2}}
\begin{proof}
Let $f(z)=z \Phi[a,b;c;q,z]$. Then $f \in \mathcal{A}$ and is of the form 
$ f(z)=z+ \displaystyle \sum_{n=2}^ \infty A_n z^n$, where 
$$A_1=1, \quad A_n= \frac{(a,q)_{n-1}(b,q)_{n-1}}{(c,q)_{n-1}(q,q)_{n-1}},\quad \mbox{ for $n \geq 2$}.
$$
From the definition of $A_n$, we observe the recurrence relation:
$$A_{n+1}= \frac{(1-aq^{n-1})(1-bq^{n-1})}{(1-cq^{n-1})(1-q^n)}A_n.
$$
First, we need to show that $\{((1-q^n)/(1-q))A_n\}$ is a decreasing sequence of positive real numbers. For this, we compute
\begin{eqnarray*}
&& \hspace*{-1cm}\left(\frac{1-q^n}{1-q} \right)A_n - \left(\frac{1-q^{n+1}}{1-q}\right)A_{n+1}\\
&=& \left(\frac{1-q^n}{1-q} \right)A_n - \left(\frac{1-q^{n+1}}{1-q}\right) \frac{(1-aq^{n-1})(1-bq^{n-1})}{(1-cq^{n-1})(1-q^n)}A_n\\
&=& \frac{A_n}{(1-cq^{n-1})(1-q^n)(1-q)}[(1-q^n)^2(1-cq^{n-1})-(1-q^{n+1})(1-aq^{n-1})(1-bq^{n-1})]\\
&=& \frac{A_n}{\left(\displaystyle\frac{1-cq^{n-1}}{1-q}\right)\left(\displaystyle\frac{1-q^n}{1-q}\right)}X(n)\\
\end{eqnarray*}
where 
$$ X(n)=\frac{1}{(1-q)^3}\left[(1-q^n)^2(1-cq^{n-1})-(1-q^{n+1})(1-aq^{n-1})(1-bq^{n-1})\right] .
$$
On simplification, we have
\begin{eqnarray*}
 X(n) &=& q^{n-1}\left\{\left(\frac{1-q^n}{1-q}\right)^2 \left(\frac{ab-c}{1-q}\right)
+\left(\frac{1-q^n}{1-q}\right)\left(\frac{aq+bq-q-2ab+ab/q}{(1-q)^2}\right)\right.\\
 &&\hspace*{9cm}\left.+\left(\frac{a+b-q-ab/q}{(1-q)^2}\right)\right\}.
\end{eqnarray*}
Therefore, to prove the first part, it is sufficient to show that $X(n)$ is non-negative. 
Note that the condition (\ref{eqn1}) implies $c\le ab$ and so the coefficient of the factor 
$((1-q^n)/(1-q))^2$ in the above 
expression of $X(n)$ is non-negative. Thus, for all $n \ge 1$, we can write
\begin{eqnarray*}
X(n) &\ge & q^{n-1}\left[\left(2\left(\frac{1-q^n}{1-q}\right)-1\right)\left(\frac{ab-c}{1-q}\right)
+\left(\frac{1-q^n}{1-q}\right)\left(\frac{aq+bq-q-2ab+ab/q}{(1-q)^2}\right)\right.\\
&&\hspace*{9.5cm}\left. +\left(\frac{a+b-q-ab/q}{(1-q)^2}\right)\right]\\
&=& q^{n-1}\left[\left(\frac{1-q^n}{1-q}\right)\left\{2\left(\frac{ab-c}{1-q}\right)
+\left(\frac{aq+bq-q-2ab+ab/q}{(1-q)^2}\right)\right \}\right.\\
&&\hspace*{5cm}\left. +\left(\frac{a+b-q-ab/q}{(1-q)^2}\right)-\left(\frac{ab-c}{1-q}\right)\right]=Y(n), \mbox{ say.}
\end{eqnarray*}
By equation (\ref{eqn1}), we have $c\le ab+(aq+bq-q-2ab+ab/q)/(2(1-q))$. So, the coefficient 
of $(1-q^n)/(1-q)$ in the expression of $Y(n)$ is non-negative and hence we obtain
$$X(n)\ge Y(n)\ge Y(1)=\left( \frac{ab-c}{1-q}\right)+\left(\frac{aq+bq-q-2ab+ab/q}{(1-q)^2}\right)
+\left(\frac{a+b-q-ab/q}{(1-q)^2}\right).
$$
Again, by (\ref{eqn1}), we get $Y(1)\ge 0$. This argument proves that 
if $c\le T_1(a,b)$ then the function $z\Phi[a,b;c;q,z]\in \mathcal{K}_q$ with the starlike function $g(z)=z/(1-z)$.

To prove the second part, we need to show that $((1-q^n)/(1-q))A_n$ is a non-decreasing sequence 
and has a limit less than or equal to $2$. From (\ref{eqn2}), we note that $c=ab$ and $ab\ge (aq+bq-q)/(2-q^{-1})$. So, by the hypothesis (\ref{eqn2}), we obtain
$$X(n)= Y(n)\le Y(1)=\left(\frac{aq+bq-q-2ab+ab/q}{(1-q)^2}\right)+\left(\frac{a+b-q-ab/q}{(1-q)^2}\right)\le 0.
$$
Now, we have to show that the limiting value of $A_n(1-q^n)/(1-q)$ 
is less than or equal to $2$.  
Write $c=ab$ and
\begin{eqnarray*}
\left(\frac{1-q^n}{1-q}\right) A_n 
&=& \left( \frac{1-q^n}{1-q}-1\right) A_n +A_n\\
&=& \frac{q(1-q^{n-1})}{1-q}\frac{(a,q)_{n-1}(b,q)_{n-1}}{(c,q)_{n-1}(q,q)_{n-1}} 
+ \frac{(a,q)_{n-1}(b,q)_{n-1}}{(c,q)_{n-1}(q,q)_{n-1}}\\
&=& \frac{q}{1-q}\frac{(a,q)_{n-2}(1-aq^{n-2})(b,q)_{n-2}(1-bq^{n-2})}{(c,q)_{n-2}(1-cq^{n-2})(q,q)_{n-2}} 
+ \frac{(a,q)_{n-1}(b,q)_{n-1}}{(c,q)_{n-1}(q,q)_{n-1}}.\\
\end{eqnarray*}
Taking limit as $n\to \infty$ on both the sides, we have
$$\lim_{n \to \infty}\left(\frac{1-q^n}{1-q}\right) A_n= \frac{q}{1-q} \lim_{n \to \infty} 
\frac{(a,q)_{n-2}(b,q)_{n-2}}{(c,q)_{n-2}(q,q)_{n-2}}+ \lim_{n \to \infty} \frac{(a,q)_{n-1}(b,q)_{n-1}}{(c,q)_{n-1}(q,q)_{n-1}} .
$$
From Remark~\ref{rem1}, we have
\begin{eqnarray*}
\lim_{n \to \infty}\left(\frac{1-q^n}{1-q}\right) A_n
&=&\frac{q}{1-q}(1-q)^{\displaystyle\log_q c- \log_q a- \log_q b+1} 
\frac {\Gamma_q (\log_q c)}{\Gamma_q (\log_q a) \Gamma_q (\log_q b)}\\
&& \hspace*{1cm}+(1-q)^{\displaystyle\log_q c-\log_q a-\log_q b+1}
\frac {\Gamma_q (\log_q c)}{\Gamma_q (\log_q a) \Gamma_q (\log_q b)}.
\end{eqnarray*}
Using $c=ab$, the above expression reduces to
\begin{eqnarray*}
\lim_{n \to \infty}\left(\frac{1-q^n}{1-q}\right) A_n 
&=& q\frac {\Gamma_q (\log_q {ab})}{\Gamma_q (\log_q a) \Gamma_q (\log_q b)}
+(1-q)\frac {\Gamma_q (\log_q {ab})}{\Gamma_q (\log_q a) \Gamma_q (\log_q b)}\\
&=& \frac{\Gamma_q(\log_q {ab})}{\Gamma_q(\log_q a)\Gamma_q(\log_q b)} .
\end{eqnarray*}
The conclusion follows from (\ref{eqn2}) and Lemma \ref{lem1}.
\end{proof}
\begin{corollary}
Let $a,b<1/q$ and $(1-a)(1-b)\ne 0$. If $c$ satisfies either $c\le \displaystyle \frac{1}{q}T_1(aq,bq)$ 
where $T_1(a,b)$ is defined in Theorem~\ref{thm2}, or
$$c=abq \mbox{ with } abq\ge \max\left \{\frac{aq+bq-1}{2-(1/q)},\frac{aq+bq+a+b-2}{2}\right \}
$$
$$\mbox{ and } \frac{\Gamma_q(\log_q abq^2)}{\Gamma_q(\log_q aq)\Gamma_q(\log_q bq)}\le 2
$$
then $z(D_q \phi)(z)$ is $q$-close-to-convex in $\mathbb{D}$, where $\phi(z)=\Phi[a,b;c;q,z]$.
\end{corollary}
\begin{proof}
Some simple calculation gives the $q$-differentiation of $\phi(z)$ in the following form: 
$$\frac{(1-a)(1-b)}{(1-c)(1-q)}z \Phi[aq,bq;cq;q,z]=z (D_q \phi)(z)
$$
i.e.
$$z\Phi[aq,bq;cq;q,z]=\frac{(1-c)(1-q)}{(1-a)(1-b)} z (D_q \phi)(z).
$$
Apply this identity in Theorem \ref{thm2} and deduce that the function $z\Phi[aq,bq;cq;q,z]$ is in 
$\mathcal{K}_q$ with the starlike function $z/(1-z)$. Therefore, the conclusion of our corollary 
follows.
\end{proof}

\section{Conclusion and future directions} 

Visualization of the $q$-theory in geometric function theory was first introduced in 1990. It has provided 
important insight into the existing function theoretic structure as well as number of problems in the current
avenues in special functions. Since 1990, apart from the works in \cite{Ro92, RS12} and the recent work in 
\cite{SS14}, there are no more investigation made in this direction. 
Therefore, we do expect that this series of visualization and working in this area will help
many researchers into networking with a growing infrastructure by illustrating interesting problems 
in this theory. The results in this manuscript also demonstrate that the computational framework in this
direction helps to generate functions having interesting geometric properties.
Further work in this direction will certainly bring a strong foundation between 
$q$-theory and geometric function theory. 
It also opens up several avenues for future work which may lead to interesting dissertations. 

One possible future direction is to generate functions having interesting geometric properties
and visualize their behavior by making 2D and 3D graphical plots.
One strong advantage to our work is that readers find interests to investigate 
the $q$-theory and its applications more in geometric function theory.
The disadvantage is that analytical problems in this direction are more difficult to
handle. However, it can be a challenge to describe the relevant image domains and find interesting
problems to work in this direction. For example, image of the unit disk under the mapping
$zF(0,2;c;z)/F(-1,2;c;z)$ converges to the unit disk when $c$ is larger and larger (see Figures~\ref{ParametricPlotH2-3}).
\begin{figure}[H]
\begin{minipage}[b]{0.5\textwidth}
\includegraphics[width=7cm]{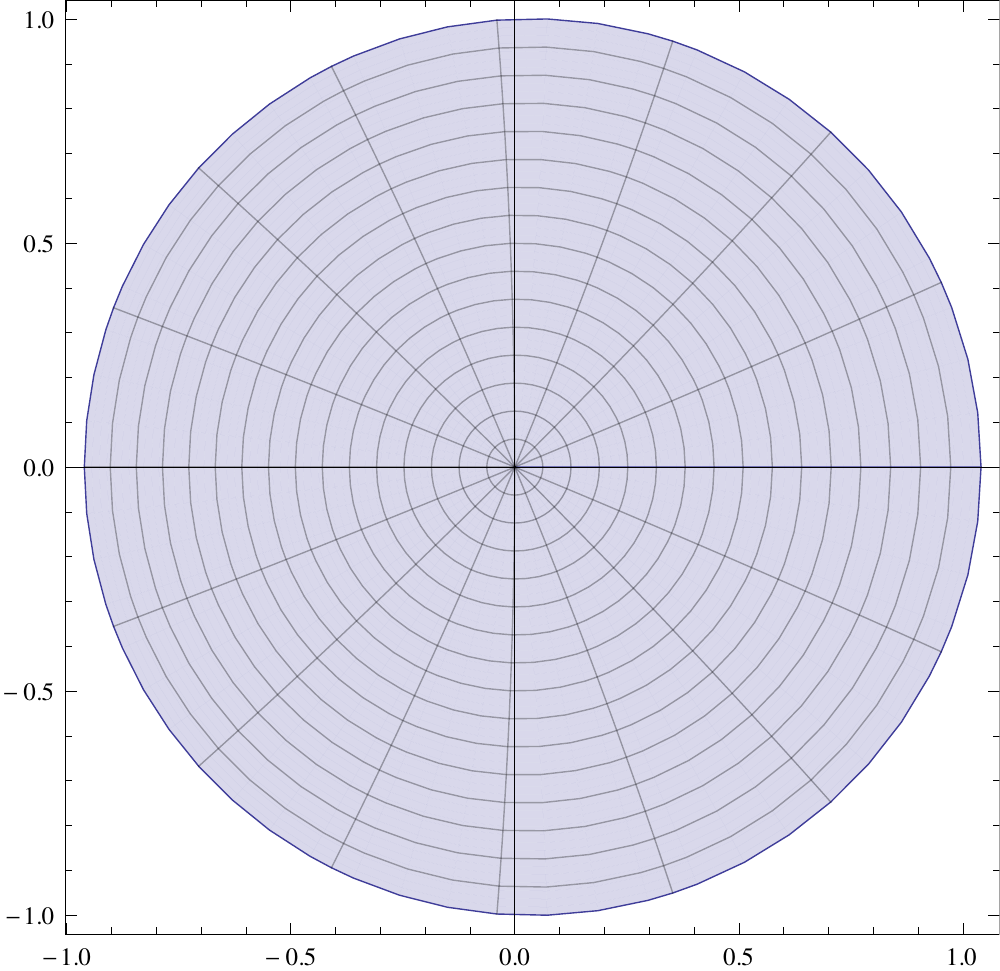}
\end{minipage}
\begin{minipage}[b]{0.45\textwidth}
\includegraphics[width=7cm]{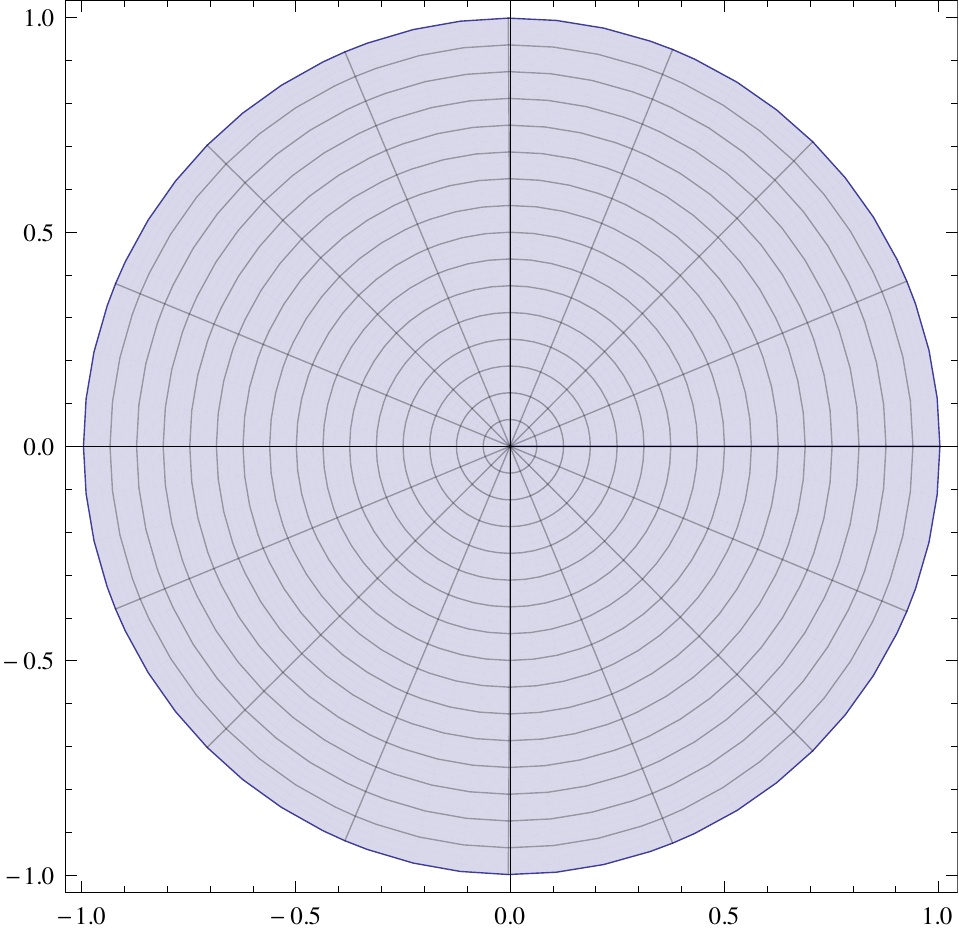}
\end{minipage}
\caption{Description of $zF(0,2;c;z)/F(-1,2;c;z)$ that maps the disk $|z|<0.999$ to a region
close to the unit disk. The left region is computed for $c=50$ and the right region is for 
$c=500$. }\label{ParametricPlotH2-3}
\end{figure}
Can it be practically possible to analyze its behaviour? 
Appropriate visualization will be extremely valuable if some application of this development 
can be worked it out connecting to quantum analysis and physics, see for instance \cite{And74,Ern02,Fin88}.

\vskip 1cm
\noindent
{\bf Acknowledgements.} The work of the first author is supported by University
Grants Commission, New Delhi (grant no. F.2-39/2011 (SA-I)). The authors thank Professor S. Ponnusamy 
for bringing the articles \cite{IMS90,PV01} to their attention and helpful
discussion on this topic.

\end{document}